\documentclass[a4paper,11pt]{amsart}
\usepackage{amsthm,amssymb,latexsym, amsmath}
\input amssym.def

\title[Logarithmic mean oscillation on the polydisc and paraproducts]{Logarithmic mean oscillation on the polydisc,
multi-parameter paraproducts and iterated commutators}
\newtheorem{theorem}{T{\hskip 0pt\footnotesize\bf HEOREM}}[section]
\newtheorem{lemma}[theorem]{L{\hskip 0pt\footnotesize\bf EMMA}}
\newtheorem{proposition}[theorem]{P{\hskip 0pt\footnotesize\bf ROPOSITION}}
\newtheorem{definition}[theorem]{D{\hskip 0pt\footnotesize\bf EFINITION}}
\newtheorem{corollary}[theorem]{C{\hskip 0pt\footnotesize\bf OROLLARY}}
\newtheorem{remark}[theorem]{R{\hskip 0pt\footnotesize\bf EMARK}}

\def\eins{\mathbf{1}}
\def\O{\Omega}

\def\D{\mathcal D}
\def\RR{\mathcal R}
\def\J{\mathcal J}

\def\RR{\mathcal{R}}

\def\vda{\vec{\delta}}
\def\Da{\Delta}

\def\veps{\vec{\varepsilon}}

\def\vba{\vec{\beta}}

\def\vj{\vec{j}}
\def\vk{\vec{k}}
\newcommand{\supp}{\mathrm{supp}}

\def\BMO{\mathrm{BMO}}
\def\LMO{\mathrm{LMO}}
\def\bmo{\mathrm{bmo}}

\newcommand{\bprop} {\begin{proposition}}
\newcommand{\eprop} {\end{proposition}}
\newcommand{\btheo} {\begin{theorem}}
\newcommand{\etheo} {\end{theorem}}
\newcommand{\blem} {\begin{lemma}}
\newcommand{\elem} {\end{lemma}}
\newcommand{\bcor} {\begin{corollary}}
\newcommand{\ecor} {\end{corollary}}

\newcommand{\Be}{\begin{equation}}
\newcommand{\Ee}{\end{equation}}
\newcommand{\Bea}{\begin{eqnarray}}
\newcommand{\Eea}{\end{eqnarray}}
\newcommand{\Bes}{\begin{equation*}}
\newcommand{\Ees}{\end{equation*}}
\newcommand{\Beas}{\begin{eqnarray*}}
\newcommand{\Eeas}{\end{eqnarray*}}
\newcommand{\Ba}{\begin{array}}
\newcommand{\Ea}{\end{array}}
\def\R{\mathbb{R}}

\def\N{\mathbb{N}}

\def\T{\mathbb{T}}
 \scrollmode

\begin{document}

\author[Beno\^it F. Sehba]{Beno\^it F. Sehba}
\address{Beno\^it Florent Sehba, D\'epartement de Math\'ematiques, Facult\'e des Sciences, Universit\'e de Yaound\'e I, B. P. 812 Yaound\'e,Cameroun. }
\email{bsehba@gmail.com}


\begin{abstract}
We introduce another notion of bounded logarithmic mean oscillation in the $N$-torus and give an equivalent definition in terms of boundedness of  multi-parameter paraproducts from the dyadic little $\BMO$, $\bmo^d(\T^N)$ to the dyadic product $\BMO$ space, $\BMO^d(\T^N)$. We also obtain a sufficient condition for the boundedness of the iterated commutators from the subspace of $\bmo(\R^N)$ consisting of functions with support in $[0,1]^N$ to $\BMO(\R^N)$.

\end{abstract}
\maketitle
\section{Introduction}
Multi-parameter paraproducts are among the most studied operators in modern harmonic analysis as they appear as the building blocks of several other operators. Their study has been at the origin of several results in the literature and it is still attracting a lot of attention \cite{bp,ferglac, laceyjason, laceyterw, muscalu, pottsehba1}.
\vskip .2cm
Our first interest in this paper will be essentially for the boundedness of dyadic paraproducts from the dyadic little $\BMO(\T^N)$, $\bmo^d(\T^N)$ of Cotlar and Sadosky \cite{cotsad} to the dyadic product $\BMO(\T^N)$ of Chang and Fefferman \cite{ChFef2}, $\BMO^d(\T^N)$.
We will focus on the so-called main  paraproduct denoted below
by $\Pi$. We prove a characterization of boundedness of $\Pi$ from $\bmo^d(\T^N)$ to $\BMO^d(\T^N)$ in terms of a new notion of logarithmic mean oscillation in the polydisc. In the two-parameter case, this notion is in fact implicit in \cite{pottsehba1} and \cite{benoit}.
Some other notions of logarithmic mean oscillation were discussed in \cite{pottsehba1, pottsehba2, benoit}. The notion of logarithmic mean oscillation that we are considering in this paper is the natural one when the first space is the little space of functions of bounded mean oscillation, the target space being the product $\BMO$.

Our second interest is for the boundedness of the iterated commutators with the Hilbert transforms from $\bmo([0,1]^N)$ to $\BMO(\R^N)$. We prove that a sufficient condition for these commutators to be bounded is given by our notion of logarithmic oscillation adapted to $\R^N$. This last interest is in the scope of the works \cite{ferglac, fergsad, pottsehba1} which can be seen as a motivation for this paper.

Our presentation is close to the one of \cite{pottsehba1}. In the next section, we provide various definitions and notations, and we give the statement of the result for the main paraproduct. Section 3 is devoted to the proof of the boundedness of the main paraproduct from $\bmo^d(\T^N)$ to $\BMO^d(\T^N)$. In section 4, we deal with the other paraproducts. The study of iterated commutators is in section 5 where for simplicity of the presentation, we will mainly discuss the two-parameter case as the $N$-parameter case follows the same steps.

\vspace{0.5cm}

\section{preliminaries and main result}
 As usual, $\mathcal D$ will be
the set of all dyadic intervals of the unit circle $\T$ that we identify with the interval $[0,1)$. The set of all dyadic rectangles
 $R=R_1\times\cdots\times R_N$, where $R_j\in \mathcal D$,
$j=1,\cdots, N$ is denoted $\mathcal R=\mathcal {D}^N$. The Haar wavelet adapted
to the dyadic interval $I$ is given by
$$
h_I=|I|^{-1/2}(\chi_{I^+}-\chi_{I^-}),
$$
where $I^+$ and $I^-$
are the right and left halves of $I$, respectively.

 The product Haar wavelet $h_R$ adapted to the rectangle $R=R_1\times\cdots\times R_N\in \RR$ is defined by $h_R(t_1,\cdots, t_N)=h_{R_1}(t_1)\cdots h_{R_N}(t_N)$.
We denote by $L_0^2(\T^N)$ the subspace of $L^2(\T^N)$ defined as follows
$$
L_0^2(\T^N) = \left\{ f \in L^2(\T^N):  \int_\T f(\cdots,t_j,\cdots) dt_j =0, j=1,\cdots, N \text{ for a.e. }t_1,\cdots, t_N \in \T \right\},
$$
so that
$$f=\sum_{R\in \mathcal {R}} \langle f,h_R \rangle h_R \qquad (f \in L^2_0(\T^N)).$$
The Haar coefficient $\langle f, h_R \rangle$ will be quite often denoted $f_R$ or $f_{S\times T}$ whenever $R=S\times T$. The mean of $f\in
L^2(\mathbb {T}^N)$ over the dyadic rectangle $R$ is denoted $m_Rf$.

\vskip .2cm
The dyadic product Hardy space $H^{1}_d(\mathbb T^N)$ is defined by

$$
H_d^{1}(\mathbb T^N)=\{f\in L_0^1(\mathbb T^N): \mathcal {S}[f]\in
L^1(\mathbb T^N)\},
$$
where $\mathcal S$ is the dyadic square function,
\begin{equation}   \label{eq:sq}
     \mathcal {S}[f] = \left(\sum_{R \in \RR}  \frac{\chi_R}{|R|} |f_R|^2 \right)^{1/2}.
\end{equation}

The dual space of the dyadic product Hardy space $H^{1}_d(\mathbb T^N)$ is the space of functions
 of dyadic bounded mean oscillations in $\mathbb T^N$,
$\BMO^d(\mathbb T^N)$ (see e.g.
\cite{bp, ChFef1, treil}) and it consists of all functions $f\in L_0^2(\mathbb
T^N)$ such that
\begin{equation}
 ||f||_{\BMO^d(\T^N)}^2:=\sup_{\O\subset \mathbb
{T}^N}\frac{1}{|\O|}\sum_{R\subset \O}|f_R|^2=\sup_{\O\subset \mathbb
{T}^N}\frac{1}{|\O|}||P_\O f||_{L^2(\T^N)}^2 < \infty,
\end{equation}
where the supremum
is taken over all open sets $\O\subset \mathbb T^N$ and $P_\O$ is the
orthogonal projection on the subspace spanned by Haar functions
$h_R$, $R\in \mathcal R$ and $R\subset  \O$.

\vskip 2.cm

The little $\BMO$ of Cotlar and Sadosky is defined by
\begin{equation} \label{bmo1}
   \bmo(\T^N)= \{ f \in L^2(\T^N): ||f||_{*,N} < \infty \},
\end{equation}
where
$$||f||_{*,N}:=\sup_{R \subset \T^N, \text{ rectangle }} \frac{1}{|R|}\int_R |f(s,t) - m_R f| ds dt$$
with $m_Rf=\frac{1}{|R|}\int_R f(t_1,\cdots, t_N) dt_1\cdots dt_N$.

The dyadic $\bmo(\T^N)$ denoted $\bmo^d(\T^N)$ is defined as above by taking the supremum only over dyadic rectangles.
\vskip .2cm
Let us introduce some further notions in the dyadic setting. For $\vj =(j_1,\cdots, j_N) \in \N_0 \times\cdots\times \N_0=\N_0^N$ we define the $j_1$-th generation of dyadic intervals and the $\vj$-th generation of dyadic rectangles as follows.
$$
     \D_{j_1} = \{I \in \D: |I|= 2^{-j_1} \},
$$
$$
    \RR_{\vj} = \D_{j_1} \times\cdots\times \D_{j_N} = \{ I_1 \times\cdots\times I_N \in \RR: |I_k|= 2^{-j_k}\}.
$$
We will be also using the following notations
$$
\D^K=\D\times\cdots\times \D\,\,\, (\textrm{K-factors}),\,\,\,K\in \N_0,
$$
for $\vj\in \N_0^K$,
$$
\D_{\vj}^K=\D_{j_1} \times\cdots\times \D_{j_K}.
$$
%
The product Haar martingale differences are given by
$$
   \Da_{\vj} f = \sum_{R\in \RR_{\vj}} \langle  f, h_R \rangle h_R,
$$
and the expectations are defined by
$$
   E_{\vj} f = \sum_{\vk \in \N_0^N ,  \vk < \vj } \Da_{\vk}f,
$$
where we write $ \vk < \vj $ for $k_l < j_l$, $l=1,\cdots,N$ and correspondingly $\vk \le \vj $ for $k_l \le j_l$,
for $f \in L^2(\T^N)$.

If we care about the variable on which we are acting, then we need the following operators
$$
   E^{(l)}_{\vj} f = \sum_{\vk  \in \N_0 \times \N_0,  k_l <j_l} \Da_{\vk}f,
$$
for $f \in L^2(\T^N)$.

The following operators defined on $L^2(\T^N)$ will be also needed
\begin{equation}   \label{eq:qdef}
   Q_{\vj}f =\sum_{ \vk \ge \vj } \Da_{\vk}f.
\end{equation}
The operators $Q^{(l)}_{\vj}$ are defined as for $E^{(l)}_{\vj}$.

\vskip 2.cm

For $I$ a dyadic interval and $\varepsilon\in \{0,1\}$, we define $h^{\varepsilon}_I$ by
$$h^{\varepsilon}_I =\left\{ \begin{matrix} h_I &\text{if }& \varepsilon=0\\
      |I|^{-1/2}|h_I| & \text{ if } & \varepsilon=1.
                                  \end{matrix} \right.
$$

For $R=R_1\times\cdots \times R_N\in \RR$ and $\vec {\varepsilon}=(\varepsilon_1,\cdots,\varepsilon_N)$, with $\varepsilon_j\in \{0,1\}$, we write
$$h^{\vec {\varepsilon}}_R(t)=h^{\varepsilon_1}_{R_1}(t_1)\cdots h^{\varepsilon_N}_{R_N}(t_N),\,\,\,t=(t_1,\cdots,t_N).$$
Several operators appearing in Fourier analysis are related to the following family of operators
\begin{equation}\label{paraprodgene2}B_{\veps,\vda,\vba}(\phi,f):=\sum_{R\in \RR}\langle \phi,h^{\vec {\varepsilon}}_R\rangle \langle f,h^{\vec {\delta}}_R\rangle h^{\vec {\beta}}_R.\end{equation}

The paraproducts we are interested in here are of the above form. They correspond to the operators
$B_{\veps,\vda,\vba}(\phi, \cdot)$ with symbol $\phi$ corresponding to triples $(\veps, \vda, \vba)$ such that $\vec {\varepsilon}=(0,\cdots,0)$ and
$$
\delta_j=\left\{ \begin{matrix} 1 &\text{if}& \beta_j=0\\

                0 &\text{otherwise}.&
                              \end{matrix}\right.
$$
For simplicity, we  denote these paraproducts by $\Pi^{\vec {\beta}}$. We will be using the notations
$\vec{1}=(1,\cdots,1)$, $\vec{0} = (0,\cdots,0)$.

\vspace{0.5cm}

Let $\phi \in L^2(\T^N)$. The (main) paraproduct $\Pi_\phi$ is defined by
$$
   \Pi_\phi f = \Pi(\phi,f):= \sum_{\vj \in \N_0 \times\cdots\times \N_0}
(\Delta_{\vj} \phi) (E_{\vj} f) = \sum_{R\in \RR} h_R \phi_R m_R f
$$
on functions with finite Haar expansion; it is just the paraproduct $\Pi^{(0,\cdots,0)}$ above.

Next, we define the space of functions of dyadic logarithmic mean oscillation on $\T^N$, $\LMO^d(\T^N)$.
\begin{definition}   \label{def:LMOd}
Let $\phi \in L^2(\T^N)$. We say that $ \phi \in \LMO^d(\T^N)$, if there exists $C >0$ with
$$
    \|Q_{\vj} \phi\|_{\BMO^d(\T^N)} \le C \frac{1}{(\sum_{k=1}^Nj_k) +N}
$$
for all $\vj = (j_1,\cdots, j_N)  \in \N_0^N$. The infinimum of such constants is denoted by $\|\phi\|_{\LMO^d(\T^N)}$.
\end{definition}

We have the following alternative characterization of our space.
\begin{proposition} \label{prop:LMOequivchar}
Let $\phi \in L^2(\T^N)$. Then $ \phi \in \LMO^d(\T^N)$, if and only if there exists a constant $C >0$ such that for each
dyadic rectangle $R= I_1 \times\cdots\times I_N$ and each open set $\Omega \subseteq R$,
\begin{equation}   \label{eq:lmochar}
   \frac{\left(\log\frac{4}{|I_1|}+\cdots +\log\frac{4}{|I_N|}\right)^2}{|\Omega|} \sum_{Q \in \RR, Q \subseteq \Omega} |\phi_Q|^2 \le C.
\end{equation}
\end{proposition}

\proof Let $\phi \in \LMO^d(\T^N)$ in the sense of Definition \ref{def:LMOd},  let $R= I_1 \times\cdots\times I_N$ be a dyadic rectangle with $|I_j| = 2^{-k_j}$, $j=1,\cdots,N$,
 and let $\Omega \subseteq R$ be open. Let $\vk = (k_1,\cdots, k_N)$. Then
 \begin{multline*}
      \sum_{Q \in \RR, Q \subseteq \Omega} |\phi_Q|^2 = \|P_\Omega \phi \|_{L^2(\T^N)}^2 = \|P_\Omega Q_{\vk} \phi \|_{L^2(\T^N)}^2 \le |\Omega| \| Q_{\vk} \phi \|^2_{\BMO^d(\T^N)}    \\
      \le
       |\Omega| \frac{1}{\left(k_1 +\cdots +k_N \right)^2}    \| \phi \|^2_{\LMO^d(\T^N)} \lesssim  \left(\log\frac{4}{|I_1|}+\cdots+\log\frac{4}{|I_N|}\right)^{-2} |\O|  \| \phi \|^2_{\LMO^d(\T^N)}.
 \end{multline*}

 Conversely, suppose that $\phi \in L^2(\T^N)$ and that (\ref{eq:lmochar}) holds. Let $\vk = (k_1,\cdots, k_N) \in \N_0^N$, and let $\Omega \subseteq \T^N$ open.
Then
\begin{multline*}
   \| P_\Omega Q_{\vk} \phi \|_{L^2(\T^N)}^2 = \sum_{R \in \RR_{\vk}}        \| P_{R \cap \Omega} Q_{\vk} \phi \|_{L^2(\T^N)}^2 \\
       \lesssim C \frac{1}{\left(k_1\cdots+k_N+N\right)^2}  \sum_{R \in \RR_{\vk}}  |R \cap \Omega|  = C \frac{1}{\left(k_1 +\cdots+k_N +N\right)^2} |\O|.
\end{multline*}
This holds for all $\Omega \subseteq \T^N$ open, hence $\|Q_{\vk} \phi \|_{\BMO^d(\T^N)} \lesssim \frac{1}{\left(k_1 +\cdots+k_N +N\right)}$.
\qed

We would like to observe an important equivalent definition of the above space, $\LMO^d(\T^N)$ . For this we introduce further definitions.
\begin{definition}
Let $\phi \in L^2(\T^N)$, $j\in \{1,\cdots,N\}$. We say that $ \phi \in \LMO_{j}^d(\T^N)$, if there exists $C >0$ with
$$
    \|Q^{(j)}_{i} \phi\|_{\BMO^d(\T^N)} \le C \frac{1}{i+1 }
$$
for all $i \in \N_0$.

The infimum of such constants is denoted by $\|\phi\|_{\LMO_{j}^d(\T^N)}$.
\end{definition}

It is not hard to see that
\begin{equation}\label{eq:LMOLMO_j}
\LMO^d(\T^N)=\cap_{j=1}^N\LMO_j^d(\T^N).
\end{equation}
One way to see this is by considering the following equivalent definition of $\LMO_j^d(\T^N)$.
\begin{proposition} \label{prop:LMOequivchar1}
Let $\phi \in L^2(\T^N)$ and $j=1,2,\cdots,N$. Then
$ \phi \in \LMO_j^d(\T^N)$, if and only if there exists $C >0$ such that for each
dyadic rectangle $R= I_1 \times\cdots \times I_N$ and each open set $\Omega \subseteq R$,
\begin{equation}   \label{eq:LMOequivchar1}
   \frac{\left(\log\frac{4}{|I_j|}\right)^2}{|\Omega|} \sum_{Q \in \RR, Q \subseteq \Omega} |\phi_Q|^2 \le C.
\end{equation}
\end{proposition}
\vskip .2cm
Here is one of our main results that gives an equivalent definition of $\LMO^d(\T^N)$ in terms of boundedness of the main paraproduct $\Pi$ from $\bmo^d(\T^N)$ to $\BMO^d(\T^N)$.

\begin{theorem}   \label{thm:main}
Let $\phi \in L^2(\T^N)$. Then $\phi \in \LMO^d(\T^N)$, if and only if $\Pi_\phi:\bmo^d(\T^N) \rightarrow \BMO^d(\T^N)$ is bounded. Moreover,
$$\|\Pi_\phi\|_{\bmo^d(\T^N) \to \BMO^d(\T^N)} \approx \|\phi\|_{\LMO^d(\T^N)}.$$
\end{theorem}

\section{The main paraproduct}
The aim of this section is to prove Theorem \ref{thm:main}.
We start by introducing some further notations.

Given an integrable function $f$ on $\T^N$ and rectangles $Q\subset \T^{N_1}$ and $S\subset \T^{N_2}$, $N=N_1+N_2$, we write $m_Qf=\frac{1}{|Q|}\int_Qf(s,t)ds$, $m_Sf=\frac{1}{|J|}\int_Jf(s,t)dt$, $s\in \T^{N_1},t\in \T^{N_2}$ and, $m_Rf=\frac{1}{|R|}\int_Rf(s,t)dsdt=\frac{1}{|R|}\int_Rf(t_1,\cdots, t_N)dt_1\cdots dt_N$, $t_j\in \T$, $j=1,\cdots,N$,
$R=Q\times S$.

We remark that if $Q$ is a rectangle in the $N_1$ first variables, then $m_Qf$ is in fact a function of the last $N-N_1$ variables. It is not hard to see that the space $\bmo(\T^N)$ has the following property.
\bprop\label{prop:meanofbmo}
Let  $f\in \bmo(\T^N)$. Then for any $R\in \T^K$, $N>K\in \N_0$, $m_Rf\in \bmo(\T^{N-K})$ uniformly. Moreover,
$$\|m_Rf\|_{\bmo(\T^{N-K})}\lesssim \|f\|_{\bmo(\T^{N})}.$$
%
\eprop

Let us also observe the following.
\blem\label{lem:testfunct}
The following assertions hold.
\begin{itemize}
\item[(1)] Given an interval $I$ in $\T$, there is a function in $\BMO(\T)$, denoted $\log_I$ such that
\begin{itemize}
\item the restriction of $\log_I$ to $I$ is $\log\frac{4}{|I|}$.
\item $\|\log_I\|_{\BMO(\T)}\le C$ where $C$ is a constant that does not depend on $I$.
\end{itemize}
\item[(2)] For any $f_j\in \BMO(\T)$, $j=1,\cdots,N$, the function $b(t_1,\cdots,t_N)=\sum_{j=1}^Nf_j(t_j)$ belongs to $\bmo(\T^N)$. Moreover,
$$\|\sum_{j=1}^Nf_j(t_j)\|_{\bmo(\T^N)}\le \sum_{j=1}^N\|f_j\|_{\BMO(\T)}.$$
\end{itemize}
\elem
\begin{proof}
We refer to \cite[Chapter 3]{benoit} for the proof of $(1)$. Assertion $(2)$
follows directly from the definition of $\bmo(\T^N)$.
\end{proof}


We will need the following.
\begin{lemma}  \label{lemma:avgrowth}  Let $b \in \bmo^d(\T^N)$,  $k\in \N_0$, $\vk = (k_1,\cdots, k_N) \in \N_0 \times\cdots\times \N_0$. Then
$$
  |m_{I\times R} b| \lesssim (k+1) \|b\|_{\bmo^d(\T^N)} \quad  (I \in \D_k, R\in \D^{N-1});
$$
$$
    \|\chi_{I\times R} b \|_{L^2(\T^N)}^2 \lesssim (k+1)^2 |I||R| \|b\|^2_{\bmo^d(\T^N)} \quad   (I \in \D_k, R\in \D^{N-1});
$$
\begin{equation}\label{eq:averg3}
    \|\chi_R P_T b \|_{L^2(\T^N)}^2 \lesssim |R| |T|  \|b\|^2_{\bmo^d(\T^N)} \quad   (R\times T \in \RR_{\vk}).
\end{equation}
\end{lemma}
\begin{proof} Let $S = I \times R$, $|I|=2^{-k}$. The one parameter estimate of the mean of a $\BMO^d(\T)$-function and the properties of functions in $\bmo^d(\T^N)$ give directly, $$|m_Sb|=|m_I(m_Rb)|\lesssim (k+1)||m_Rb||_{\BMO^d(\T)}\lesssim (k+1)||b||_{\bmo^d(\T^N)},$$ giving the first inequality.


For the second inequality,  we observe that for $L\in \D^N$, the following identity holds:
\begin{equation}\label{eq:P_Rb}
\chi_Lb= P_L b+\sum_{S\,\,\,\textrm{a factor of}\,\,\, L}\varepsilon_S\chi_L(s,t) m_S b(t)
\end{equation}
where $\varepsilon_S=(-1)^K$, $K\in \N_0$, $S\in \D^K$. Here and below, we say $S$ is a factor of $L=L_1\times \cdots \times L_N$, if $S=L_{j_1}\times \cdots \times L_{j_M}$,
$j_l\in \{1,\cdots,N\}$, $M\le N$. When $M<N$, we say $S$ is a subfactor of $L$. Hence,


$$\|\chi_Lb\|_{L^2(\T^N)}\lesssim \|P_L b(s,t)\|_{L^2(\T^N)}+\sum_{S\,\,\,\textrm{a subfactor of}\,\,\, L}\|\chi_L(s,t) m_S b(t)\|_{L^2(\T^N)}+\|\chi_L(s,t) m_L b\|_{L^2(\T^N)}.$$
Thus, we only need to estimate the first and the last terms in the right hand side of the above inequality and the norm $\|\chi_{S\times Q}(s,t) m_S b(t)\|_{L^2(\T^N)}$, with $Q\times S=L=I\times R$.
\vskip .1cm
Clearly $\|P_L b \|_{L^2(\T^N)}^2 \le |L| \|b\|^2_{\BMO^d(\T^N)}\le |I||R| \|b\|^2_{\bmo^d(\T^N)}$ and
$$\|\chi_L m_Lb \|_{L^2(\T^N)}^2 = |m_L
b|^2 |L| \lesssim (k+1)^2|I| |R| \|b\|^2_{\bmo^d(\T^N)}$$ by the first inequality in Lemma
\ref{lemma:avgrowth}. To estimate $\|\chi_{S\times Q}(s,t) m_S b(t)\|_{L^2(\T^N)}$, we suppose that $S\in \D^{N-K}$ and $Q\in \D^K$. It follows that
\begin{eqnarray*}
\|\chi_{S\times Q}(s,t) m_S b(t)\|_{L^2(\T^N)} &=& |S|^{1/2} \| \chi_Q(t) m_S b(t) \|_{L^2(\T^K)} \\
&\lesssim& |S|^{1/2}  |Q|^{1/2}\left(  \| m_S b \|_{\bmo^d(\T^K)}+|m_{I\times R}b|\right) \\
&\lesssim& |I|^{1/2}|R|^{1/2} (k+1)  \| b \|_{\bmo^d(\T^N)} \\
\end{eqnarray*}
by the first inequality and the fact that $\| m_S b(t) \|_{\bmo^d(\T^K)}\lesssim \|b\|_{\bmo^d(\T^N)}$.

\vskip .2cm

For the last inequality, we use an induction argument on the number of parameters $N$. Starting from $N=2$, we first show that for $I,J\in \D$, we have
\begin{equation}\label{eq:dim2inneq}
\|\chi_I(s)P_Jb(s,t)\|_{L^2(\T^2)}^2\lesssim |I||J|\|b\|_{\bmo^d(\T^2)}^2.
\end{equation}

Using the decomposition (\ref{eq:P_Rb}) above, we find that
$$\|\chi_I(s)P_Jb(s,t)\|_{L^2(\T^2)}^2\lesssim \|P_{IJ}b\|_{L^2(\T^2)}^2+\|\chi_I(s)P_Jm_Ib(t)\|_{L^2(\T^2)}^2.$$
We have already seen that $\|P_{IJ}b\|_{L^2(\T^2)}^2\le |I||J|\|b\|_{\bmo^d(\T^2)}$, so we only have to take care of $\|\chi_I(s)P_Jm_Ib(t)\|_{L^2(\T^2)}^2$.

We clearly obtain
\begin{eqnarray*}
\|\chi_I(s)P_Jm_Ib(t)\|_{L^2(\T^2)}^2 = |I|\|P_Jm_Ib(t)\|_{L^2(\T)}^2\le |I||J|\|m_Ib\|_{\BMO^d(\T)}\le |I||J|\|b\|_{\bmo^d(\T^2)}.
\end{eqnarray*}

Let us now suppose that the inequality holds in $(N-1)$-parameter and prove that this also holds in $N$-parameter.

From the identity (\ref{eq:P_Rb}), we get
\Beas
\|\chi_R(s,t)P_Tb(s,t,u)\|_{L^2(\T^N)} &\lesssim& \|P_{R\times T} b\|_{L^2(\T^N)}+\sum_{S\,\,\,\textrm{a subfactor of}\,\,\, R}\|\chi_R(s,t) m_S P_Tb(t,u)\|_{L^2(\T^N)}\\ & & +\|\chi_R(s,t) m_R P_Tb(u)\|_{L^2(\T^N)}.
\Eeas
We first consider the term $\|\chi_R(s,t) m_S P_Tb(t,u)\|_{L^2(\T^N)}$, we suppose that $R=S\times Q\in \D^K$, $K=K_1+K_2<N$, and $S\in \D^{K_1}$. It follows from our hypothesis that
\begin{eqnarray*}
\|\chi_R(s,t) m_S (P_T b)(t,u)\|^2_{L^2(\T^N)} &=& |S|  \|\chi_Q(t) m_S(P_T b)(t,u) \|_{L^2(\T^{N-K_1})}^2 \\
&\lesssim& |S||Q|  |T|  \| m_S b \|_{\bmo^d(\T^{N-K_1})}^2 \\
&\lesssim& |R|  |T|  \| b \|_{\bmo^d(\T^N)}^2.
\end{eqnarray*}

Next we have
\begin{eqnarray*}
\|\chi_R(s,t) m_R P_Tb(u)\|_{L^2(\T^N)}^2 &=& |R|\|m_R P_Tb(u)\|_{L^2(\T^{N-K})}^2\\ &\le& |R||T|\|m_Rb\|_{\BMO^d(\T^{N-K})}^2\\ &\le& |R||T|\|m_Rb\|_{\bmo^d(\T^{N-K})}^2\\ &\le& |R||T|\|b\|_{\bmo^d(\T^{N})}^2.
\end{eqnarray*}

The proof is complete.

\end{proof}

\begin{remark}
From the first inequality in the above lemma, we obtain that
$$
  |m_{ R} b| \lesssim \left(k_1+\cdots+k_N+N\right) \|b\|_{\bmo^d(\T^N)} \quad  (R\in \mathcal {R}_{\vec {k}}, \vec {k}\in \mathbb {N}_0^N)
$$
and this is sharp. The sharpness is obtained by testing with the function $$\log_R(t_1,\cdots,t_N)=\log_{R_1}(t_1)+\cdots+\log_{R_N}(t_N)$$
where for any interval $I\in \D$, the function $\log_I$ is given in Lemma \ref{lem:testfunct}.
\end{remark}

Next, for $k,l \in \N_0 $ and $b\in L^2(\T^N)$, we consider on $L^2(\T^N)$ the operator
$\Pi_bE_{k}^{(l)} = \Pi(b, E_{k}^{(l)}\, \cdot)$   given by
$$\Pi_bE_{k}^{(l)}f=\Pi(b,E_{k}^{(l)}f), \,\,\,f\in L^2(\T^N).
$$
The next lemma is proved in \cite{pottsehba1}.
\begin{lemma} \label{lemma:core}
Let $b \in L^2(\T^N)$ and let $k,l \in \N_0$. Then
$$
   \|\Pi_b E^{(l)}_{k}\|_{L^2(\T^N) \to L^2(\T^N)} =   \| \Pi_{\sigma^{(l)}_{k} b}\|_{L^2(\T^N) \to L^2(\T^N)},
$$
where for $R=\prod_{j=1}^NR_j=R_l\times S\in \D^N$, $S\in \D^{N-1}$,
$$
   (\sigma^{(l)}_{k}b)_{R} =\left\{ \begin{matrix} b_{R} & \text{ if } & |R_l|
   > 2^{-k}\\
 (\sum_{R_l' \subseteq R_l} |b_{R_l'\times S}|^2)^{1/2}   & \text{ if } &
   |R_l| = 2^{-k}\\
      0 & \text{ otherwise. }&
                               \end{matrix} \right.
$$
\end{lemma}

The following lemma is the bedrock of our proof.
\begin{lemma} \label{lemma:core2}  Let $\phi \in \BMO^d(\T^N)$, $b\in \bmo^d(\T^N)$, and $k,j \in \N_0$.
  Then
$$  \|\Pi\left(\Pi(\phi,b), E_{k}^{(j)}\,\,  \cdot \right)\|_{L^2(\T^N) \to L^2(\T^N)} \lesssim (k+1) \, \|\phi\|_{\BMO^d(\T^N)} \|b \|_{\bmo^d(\T^N)}.
           $$
\end{lemma}
\begin{proof}
Following Lemma \ref{lemma:core}, we have to estimate the $\BMO^d(\T^N)$ norm of $\sigma_{k}^{(j)} (\Pi_{\phi} b)=\sigma_{k}^{(j)} (\Pi(\phi, b))$. Without loss of generality, we can suppose that $j=1$. For simplicity, we remove the supscript $(1)$ and write $E_k, Q_k$ and $\sigma_k$ for $E_k^{(1)}$, $Q_k^{(1)}$ and $\sigma_k^{(1)}$ respectively. Clearly
\begin{equation*}
   \sigma_{k} ( \Pi_\phi b) =  \sigma_{k} (E_{k} \Pi_\phi b) +
    \sigma_{k}(Q_{k} \Pi_\phi b)  = I + II.
\end{equation*}

Let us  start with term I.
For any open set $\Omega \subseteq \T^N$,
\begin{eqnarray*}
    \frac{1}{|\Omega|} \|P_{\Omega} E_{k} \Pi_\phi b\|_{L^2(\T^N)}^2
&=& \frac{1}{|\Omega|} \sum_{R = R_1 \times\cdots\times R_N, |R_1| > 2^{-k},
   R \subset \Omega}  |\phi_R|^2 |m_R b|^2 \\
& \lesssim & \frac{(k+1)^2}{|\Omega|} \sum_{R_1 \times\cdots\times R_N, |R_1| > 2^{-k},
   R \subset \Omega}  |\phi_R|^2   \|b\|^2_{\bmo^d(\T^N)} \\
&\lesssim & (k+1)^2 \| \phi \|^2_{\BMO^d(\T^N)}  \|b\|^2_{\bmo^d(\T^N)}
\end{eqnarray*}
with the help of Lemma \ref{lemma:avgrowth} at the second inequality.

\vskip .2cm
To deal with term II, we observe that $\sigma_{k}( Q_{k} \Pi_\phi b)$ has only nontrivial Haar
coefficients for $R=R_1\times\cdots\times R_N \in \RR$, with $|R_1|=2^{-k}$. We first compute $\|P_R \sigma_{k}(Q_{k} \Pi_\phi b )\|_{L^2(\T^N)}^2$
for $R\in \RR$ a rectangle of this type.
\begin{eqnarray*}
  \int_R |P_R \sigma_{k}(Q_{k} \Pi_\phi b )   |^2 ds dt
&\le& \sum_{S \subseteq R}
  |\phi_{S}|^2 |m_{S} b|^2 \\
&=&  \|\Pi_\phi \chi_R b\|^2_{L^2(\T^N)} \\
& \lesssim &   \|\phi\|^2_{\BMO^d(\T^N)} \| \chi_R b\|_{L^2(\T^N)}^2 \\
&\lesssim &   (k+1)^2 |R|\| \phi\|^2_{\BMO^d(\T^N)} \|b\|^2_{\bmo^d(\T^N)}
\end{eqnarray*}
where we used Lemma \ref{lemma:avgrowth} at the last inequality.

Now for $\Omega\subseteq \T^N$ open, we denote by $\mathcal {M}_k(\Omega)$ the set of all rectangles $R=R_1\times\cdots\times R_N\in \RR,\,\,\, R\subseteq \Omega$ which are maximal in $\Omega$ with respect to $|R_1|=2^{-k}$.
If $\mathcal {M}_k(\Omega)=\emptyset$, then $$\|P_\Omega\sigma_k(Q_{k} \Pi_\phi b )\|_{L^2(\T^N)}^2=0\le (k+1)^2 |\Omega|\| \phi\|^2_{\BMO^d(\T^N)} \|b\|^2_{\bmo^d(\T^N)}.$$

If $\mathcal {M}_k(\Omega)\neq \emptyset$, then using the above estimate of the $L^2$ norm of $P_R\sigma_{k}( Q_{k} \Pi_\phi b)$ for $R=R_1\times\cdots\times R_N \in \RR$, with $|R_1|=2^{-k}$, we obtain
\begin{eqnarray*}
\|P_\Omega\sigma_k(Q_{k} \Pi_\phi b )\|_{L^2(\T^N)}^2 &=&  \sum_{R\in \mathcal {M}_k(\Omega)}\sum_{R'\subseteq R, |R'_1|=2^{-k}}|\left(\sigma_{k}( Q_{k} \Pi_\phi b)\right)_{R'}|^2\\ &\le& \sum_{R\in \mathcal {M}_k(\Omega)}\|P_R \sigma_{k}(Q_{k} \Pi_\phi b )\|_{L^2(\T^N)}^2\\ &\lesssim& (k+1)^2\| \phi\|^2_{\BMO^d(\T^N)} \|b\|^2_{\bmo^d(\T^N)}\sum_{R\in \mathcal {M}_k(\Omega)}|R|\\ &\lesssim& (k+1)^2 |\Omega|\| \phi\|^2_{\BMO^d(\T^N)} \|b\|^2_{\bmo^d(\T^N)}.
\end{eqnarray*}

The proof is complete.
\end{proof}

We deduce the following from Definition \ref{def:LMOd}, the equality (\ref{eq:LMOLMO_j}) and, Lemma \ref{lemma:core2}.
\begin{lemma} \label{lemma:core2bis}  Let $\phi \in \LMO^d(\T^N)$,  $b \in \bmo^d(\T^N)$ and $j,k,l \in \N_0$.
  Then
$$  \|\Pi\left(\Pi(Q_{j}^{(l)}\phi,b), E_{k}^{(l)}\,\,  \cdot \right)\|_{L^2(\T^N) \to L^2(\T^N)} \lesssim \frac{k + 1}{j + 1} \, \|\phi\|_{\LMO^d(\T^N)} \|b \|_{\bmo^d(\T^N)}.
           $$
\end{lemma}

\begin{proof}[ Proof of Theorem \ref{thm:main}]. We begin by proving necessity. Suppose that \\
$\Pi_\phi: \bmo^d(\T^N) \rightarrow \BMO^d(\T^N)$ is bounded.  Let $R = R_1 \times\cdots\times R_N$ be a given dyadic rectangle, and let
$\Omega \subseteq R$ be open. We take as test function, $b(t_1,\cdots,t_N)=\log_R(t_1,\cdots,t_N)=\sum_{j=1}^N\log_{R_j}(t_j)$, where for an interval $I$, the function $\log_I(x)$ is given in Lemma \ref{lem:testfunct}.
Then
\begin{eqnarray*}
    \frac{\left(\sum_{j=1}^N\log\frac{4}{|R_j|}\right)^2}{|\Omega|} \sum_{Q \in \RR, Q \subseteq \Omega} |\phi_Q|^2
 &\approx&  
     \frac{1}{|\Omega|} \sum_{Q \in \RR, Q \subseteq \Omega} |\phi_Q|^2 |m_Qb|^2\\
  &\le& \|\Pi_\phi b\|^2_{\BMO^d(\T^N)}\\ &\le& C^2 \|\Pi_\phi\|^2_{\bmo^d(\T^N) \to \BMO^d(\T^N)}.
\end{eqnarray*}
Thus $\phi \in \LMO^d(\T^N)$ by Proposition \ref{prop:LMOequivchar}, with the appropriate norm estimate.

To prove sufficiency of the $\LMO^d(\T^N)$ condition for the boundedness of the paraproduct from $\bmo^d(\T^N)$ to $\BMO^d(\T^N)$,
we recall that $\phi \in \LMO^d(\T^N)$ implies that $\phi\in \LMO_j^d(\T^N)$, $j=1,\cdots,N$. Let $\phi \in \LMO^d(\T^N)$ and $b \in \bmo^d(\T^N)$. We recall that the following estimate holds
$$ \| \Pi_\phi b \|_{\BMO^d(\T^N)} \approx \| \Pi_{\Pi(\phi,b)}\|_{L^2(\T^N) \to L^2(\T^N)}=\| \Pi\left(\Pi(\phi,b),\cdot\right)\|_{L^2(\T^N) \to L^2(\T^N)}.$$ Motivated by the equality (\ref{eq:LMOLMO_j}) and Lemma \ref{lemma:core2}, we would like to apply Cotlar's Lemma in one direction (parameter).

For $N \in \N_0$, let
\begin{equation} \label{eq:pnk}
\begin{split}
   P_{N} &= \sum_{j=2^N-1}^{2^{N+1}-2}
   \Da_{j}^{(1)}, \\
P^{N} &= \sum_{j=2^N}^\infty \Da_{j}^{(1)},
\end{split}
\end{equation}
and
$$
   T_{N} =  \Pi_{\Pi(\phi, b)} P_{N}=\Pi\left(\Pi(\phi,b), P_{N}\, \cdot\right).
$$

Our operator can be written now $ \Pi(\Pi(\phi, b), \cdot ) = \sum_{N=0}^\infty T_{N}$. Following exactly  \cite{pottsehba1} with the help of Lemma \ref{lemma:core2}, as $T_{N} T_{N'}^* =0$ for $N \neq N'$, we obtain that
$$\|T_{N}^* T_{N'}\|\le 2^{-|N-N'|} \|\phi\|^2_{\LMO_1^d(\T^N)}\|b\|^2_{\bmo^d(\T^N)}.$$
Thus, by applying Cotlar's Lemma, we obtain that $T= \Pi(\Pi(\phi, b),\cdot)$ is bounded, and there exists an absolute constant $C>0$ with
$$
\|\Pi\left(\Pi(\phi, b), \cdot \right)\|_{L^2(\T^N)\rightarrow L^2(\T^N)}\le C \|\phi\|_{\LMO^d(\T^N)}\|b\|_{\bmo^d(\T^N)}.
$$
Consequently,
$$
     \| \Pi(\phi, b)\|_{\BMO^d(\T^N)} \le C \|\phi\|_{\LMO^d(\T^N)}\|b\|_{\bmo^d(\T^N)}.
$$

\end{proof}

\section{The other paraproducts}
The other paraproducts correspond to $\vec {\beta}=\vec {1}=(1,\cdots,1)$ and $\vec {\beta}\neq \vec {0}, \vec {1}$. The first one is the adjoint of $\Pi=\Pi^{\vec {0}}$, it is defined on $L^2(\T^N)$ by
$$
\Pi^{\vec {1}}(\phi,f)=\Delta_\phi f=\Delta(\phi,f) =  \sum_{R \in \RR} \frac{\chi_R}{|R|} \phi_R f_R,
$$ then there are the mixed paraproducts
given by the following general form
$$
\Pi^{\vec {\beta}}(\phi,f)=\sum_{R=Q \times S \in \RR}\frac{\chi_Q(s)}{|Q|} h_S(t) \phi_{R} m_S f_Q,
$$
where $R_j=Q_{j}$ if $\beta_j=1$ and $R_j=S_j$ if $\beta_j=0$.

Let us introduce some further definitions and notations.
\begin{definition} \label{def:multLMO}
Let $\phi \in L^2(\T^N)$, $\vec {\delta}=(\delta_1,\cdots,\delta_N)$, $\delta_j\in \{0,1\}$. Then $ \phi \in \LMO_{\vec {\delta}}^d(\T^N)$, if and only if there exists $C >0$ such that for each
dyadic rectangle $R= R_1 \times R_2 \times \cdots \times R_N \in \D^N$ and each open set $\Omega \subseteq R $,
$$
   \frac{\left(\log(\frac{4}{|R_{\delta_1}|})+ \cdots  \log(\frac{4}{|R_{\delta_N}|})\right)^2}{|\Omega|} \sum_{Q \in \RR, Q \subseteq \Omega} |\phi_Q|^2 \le C,
$$
where
$$
R_{\delta_j}=\left\{ \begin{matrix} R_j &\text{if}& \delta_j=0\\
                    \T &\text{otherwise.}&
                    \end{matrix}\right.
$$
\end{definition}

When $\vec {\delta}=\vec{0}=(0,\cdots,0)$, $\LMO_{\vec {\delta}}^d(\T^N)=\LMO^d(\T^N)$. One easily sees that for
$\vec {\delta}=(1,\cdots,1)= \vec{1}$, the corresponding space is just the space $\BMO^d(\T^N)$. We observe that
$$\LMO_{\vec {\delta}}^d(\T^N)=\bigcap_{j=1, \delta_j=0}^N\LMO_j^d(\T^N).$$

Let us recall with \cite{bp} the following result.
\begin{proposition}\label{prop:Deltabp}
Let $\phi \in L_0^2(\T^N)$. Then
$\Pi^{(1,\dots,1)}_\phi = \Delta_\phi: \BMO^d(\T^N) \rightarrow \BMO^d(\T^N)$ is bounded, if and only if $ \phi \in \BMO^d(\T^N)$.
Moreover, $$\|\Delta_\phi\|_{\BMO^d(\T^N) \rightarrow \BMO^d(\T^N)} \approx\|\phi\|_{\BMO^d(\T^N)}.$$
\end{proposition}

Our main result in this section is the following.
 \begin{theorem}   \label{thm:others}
Let $\phi \in L_0^2(\T^N)$, $\vec {\beta}=(\beta_1,\cdots,\beta_n)$, $\beta_j\in \{0,1\}$. Then for $\vec {\beta}\neq (0,\cdots,0),  (1,\cdots,1)$,


$\Pi^{\vec {\beta}}_\phi:\bmo^d(\T^N) \rightarrow \BMO^d(\T^N)$ is bounded if $ \phi \in \LMO_{\vec {\beta}}^d(\T^N)$.
Moreover, $$\|\Pi^{\vec {\beta}}\phi\|_{\bmo^d(\T^N) \rightarrow \BMO^d(\T^N)}
  \lesssim \|\phi\|_{\LMO_{\vec {\beta}}^d(\T^N)}.$$

\end{theorem}
 The proof of the above result requires the following lemma.


\begin{lemma} \label{lemma:core2one}  Let
 $\phi\in \BMO^d(\T^N)$, $b\in \bmo^d(\T^N)$, and $ j, k \in \N$, $\vec {\beta}\neq \vec {0}, \vec {1}$, with $\beta_j=1$. Then
$$
   \|\Pi\left(\Pi^{(\vec {\beta})}(\phi, b), E^{(j)}_{k}\, \cdot \right)\|_{L^2(\T^N) \to L^2(\T^N)} \lesssim (k+1)\|\phi\|_{\BMO^d(\T^N)} \|b \|_{\bmo^d(\T^N)}.
$$
\end{lemma}
\begin{proof} We can suppose that $j=1$ and write $E_k$ for $E_k^{(1)}$, $\sigma_k$ for $\sigma_k^{(1)}$, and $Q_k$ for $Q_k^{(1)}$.
We have to estimate
\begin{equation*}
\| \sigma_k( {\Pi^{\vec {\beta}}}(\phi, b))\|_{\BMO^d(\T^N)} \le \| \sigma_k( {\Pi^{\vec {\beta}}}({E_k\phi}, b))\|_{\BMO^d(\T^N)}
 + \| \sigma_k(  {\Pi^{\vec {\beta}}}({Q_k\phi}, b))\|_{\BMO^d(\T^N)}.
\end{equation*}

We start with the second term.
Since
${\Pi^{\vec {\beta}}}(Q_k\phi, b)$ has no nontrivial Haar terms in the first variable for rectangle  $R$ with $|R_1| > 2^{-k}$,
$$
    \sigma_k( {\Pi^{\vec {\beta}}}(Q_k\phi, b)) =
       \sum_{Q } \sum_{S=S_1\times T, |S_1|= 2^{-k}} h_S(s)  (\sum_{S_1' \subseteq S_1} |\phi_{S_1'TQ}|^2 |m_{S_1'T} b_Q|^2)^{1/2}
             \frac{\chi_Q}{|Q|}(t)
$$
and this has only nontrivial Haar terms in the first variable for rectangles $R=R_1\times\cdots\times R_N$ with
$|R_1|=2^{-k}$ (here $R_j$ is understood as an interval in $j$th variable). 
Hence we first compute $\|P_R \sigma_k( {\Pi^{\vec {\beta}}}(Q_k\phi, b))\|_{L^2(\T^N)}$ for
$R\in \RR$ of the form $R=R_1\times\cdots\times R_N=Q \times S$, with $|R_1| = 2^{-k}$, $Q\in \D^K$.

We have
\begin{eqnarray*}
   \| P_R \sigma_k( {\Pi^{\vec {\beta}}}(Q_k\phi, b))\|_{L^2(\T^N)}^2
   &=&  \| P_R \sigma_k( {\Pi^{\vec {\beta}}}({P_RQ_k\phi}, b))\|_{L^2(\T^N)}^2 \\
 &\le &  \| \sigma_k( {\Pi^{\vec {\beta}}}({P_RQ_k\phi}, b))\|_{L^2(\T^N)}^2 \\ &\le&
  \|  {\Pi^{\vec {\beta}}}({P_RQ_k\phi}, b)\|_{L^2(\T^N)}^2 \\
  &=& \|\sum_{Q' \subseteq Q} \sum_{S' \subseteq S} h_{S'} \frac{\chi_{Q'}}{|Q'|}(t)
     \phi_{S'Q'} m_{S'} b_{Q'} \|_{L^2(\T^N)}^2\\
  &=& \sum_{S' \subseteq S}\|\sum_{Q' \subseteq Q}\frac{\chi_{Q'}}{|Q'|}(t)
     \phi_{S'Q'} m_{S'} b_{Q'}\|_{L^2(\T^K)}^2\\
 &=& \sum_{S' \subseteq S}\|\Delta_{P_Q\phi_{S'}}m_{S'}b\|_{L^2(\T^K)}^2\\ &=& \sum_{S' \subseteq S}\|\Delta_{m_{S'}b}{P_Q\phi_{S'}}\|_{L^2(\T^K)}^2\\ &\lesssim& \sum_{S' \subseteq S}\|P_Q\phi_{S'}\|_{L^2(\T^K)}^2\|m_{S'}b\|_{\BMO^d(\T^K)}^2\\ &\lesssim& \|b\|_{\bmo^d(\T^N)}^2\sum_{S' \subseteq S}\|P_Q\phi_{S'}\|_{L^2(\T^K)}^2\\ &=& \|b\|_{\bmo^d(\T^N)}^2\|P_{Q\times S}\phi\|_{L^2(\T^N)}^2\\
   &\lesssim& |R|\|\phi \|_{\BMO^d(\T^N)}^2 \|b\|^2_{\bmo^d(\T^N)}.
\end{eqnarray*}

Next, for general open set $\Omega\subseteq \T^N$, we still denote by $\mathcal {M}_k(\Omega)$ the set of all rectangles $R=R_1\times\cdots\times R_N\in \RR,\,\,\, R\subseteq \Omega$ which are maximal in $\Omega$ with respect to $|R_1|=2^{-k}$.
If $\mathcal {M}_k(\Omega)=\emptyset$, then $$\|P_\Omega\sigma_k( {\Pi^{\vec {\beta}}}(Q_k\phi, b))\|_{L^2(\T^N)}^2=0\le  |\Omega|\| \phi\|^2_{\BMO^d(\T^N)} \|b\|^2_{\bmo^d(\T^N)}.$$

If $\mathcal {M}_k(\Omega)\neq \emptyset$, then using the above estimate of the $L^2$ norm of $P_R\sigma_k( {\Pi^{\vec {\beta}}}(Q_k\phi, b))$ for $R=R_1\times\cdots\times R_N \in \RR$, with $|R_1|=2^{-k}$, we obtain
\begin{eqnarray*}
\|P_\Omega\sigma_k( {\Pi^{\vec {\beta}}}(Q_k\phi, b))\|_{L^2(\T^N)}^2 &=&  \sum_{R\in \mathcal {M}_k(\Omega)}\sum_{R'\subseteq R, |R'_1|=2^{-k}}|\left(\sigma_k( {\Pi^{\vec {\beta}}}(Q_k\phi, b))\right)_{R'}|^2\\ &\le& \sum_{R\in \mathcal {M}_k(\Omega)}\|P_R \sigma_k( {\Pi^{\vec {\beta}}}(Q_k\phi, b))\|_{L^2(\T^N)}^2\\ &\lesssim& \| \phi\|^2_{\BMO^d(\T^N)} \|b\|^2_{\bmo^d(\T^N)}\sum_{R\in \mathcal {M}_k(\Omega)}|R|\\ &\lesssim&  |\Omega|\| \phi\|^2_{\BMO^d(\T^N)} \|b\|^2_{\bmo^d(\T^N)}.
\end{eqnarray*}

Let us go back to the first term $ \| \sigma_k(
  {\Pi^{\vec {\beta}}}({E_k\phi}, b))\|_{\BMO^d(\T^N)}$. Let $\Omega
  \subseteq \T^N$ be open, $R=Q\times S$, $Q\in \D^K$. $\J_S = \cup_{Q \in \D^K, Q \times S
  \subseteq \Omega } Q $ for $S \in \D^{N-K}$. Then
\begin{eqnarray*}
   \| P_\Omega \left(\sigma_k( {\Pi^{\vec {\beta}}}(E_k\phi, b))\right)\|_{L^2(\T^N)}^2
   &=&  \| P_{\O} \sigma_k ({\Pi^{\vec {\beta}}}(P_\Omega E_k\phi, b))\|_{L^2(\T^N)}^2 \\
& \le &  \| \sigma_k ({\Pi^{\vec {\beta}}}({P_\Omega E_k\phi}, b))\|_{L^2(\T^N)}^2
= \|  {\Pi^{\vec {\beta}}}({P_\Omega E_k\phi}, b)\|_{L^2(\T^N)}^2 \\
   &=& \|\sum_{S \in \D^{N-K}, |S_1| > 2^{-k}} \sum_{Q \in \D^K: S \times Q \subseteq \Omega }  h_{S}(s) \frac{\chi_{Q}}{|Q|}(t)
     \phi_{SQ} m_{S} b_{Q} \|_{L^2(\T^N)}^2\\
 &=& \sum_{S \in \D^{N-K}, |S_1| > 2^{-k}} \|\sum_{Q \subseteq \J_S } \frac{\chi_{Q}}{|Q|}(t)
     \phi_{QS} m_{S}b_{Q} \|_{L^2(\T^K)}^2\\
 &=& \sum_{S \in \D^{N-K}, |S_1| > 2^{-k}} \|\Delta_{m_S b} P_{\J_S} \phi_S \|_{L^2(\T^K)}^2\\
&\lesssim& \sum_{S \in \D^{N-K}, |S_1| > 2^{-k}} \|m_S b\|^2_{\BMO^d(\T^K)}  \|P_{\J_S} \phi_S \|_{L^2(\T^K)}^2\\
&\lesssim&  \|b\|^2_{\bmo^d(\T^N)} \sum_{S \in \D^{N-K}}  \|P_{\J_S} \phi_S \|_{L^2(\T^K)}^2\\
&\lesssim&  \|b\|^2_{\bmo^d(\T^N)}\|P_\Omega \phi \|_{L^2(\T^N)}^2\\
&\lesssim&  \|b\|^2_{\bmo^d(\T^N)} \|\phi\|^2_{\BMO^d(\T^N)}|\Omega|.
\end{eqnarray*}
\end{proof}
As in the last section, we immediately deduce that
\begin{eqnarray*}
  \|\Pi\left(\Pi^{\vec {\beta}}(Q^{(1)}_j \phi, b), E^{(1)}_{k}\, \cdot \right)\|_{L^2(\T^N) \to L^2(\T^N)} &\lesssim& \frac{k+1}{j+1} \|\phi\|_{\LMO_1^d(\T^N)} \|b \|_{\bmo^d(\T^N)}\\ &<& \frac{k+1}{j+1} \|\phi\|_{\LMO_{\vec {\beta}}^d(\T^N)} \|b \|_{\bmo^d(\T^N)}.
\end{eqnarray*}

The remainder of the proof of Theorem \ref{thm:others} is now exactly analogous to the proof of
Therem \ref{thm:main}, defining $T_N = \Pi({\Pi^{\vec {\beta}}}(\phi, \cdot), P_N\cdot)$, where
  $P_{N} = \sum_{i=2^N-1}^{2^{N+1}-2}
   \Da^{(1)}_{i}$,
and using Cotlar's Lemma in one parameter once more.

\section{Commutators and dyadic shifts}
In this section, for convenience, we restrict ourselves to the two-parameter case. Our interest here is for the iterated commutators with the Hilbert transforms. We would like to mention some facts that can also be seen as a motivation for this paper.
We will be writing $H_1$ and $H_2$ for the Hilbert transform in the first and second variables respectively. The first fact is a result of Ferguson-Lacey-Sadosky.
\begin{theorem}[\cite{ChFef1, fergsad, ferglac}]    \label{thm:fact}    Let $\phi
  \in \BMO(\R^2)$. Then
$$
   [H_1, [H_2, \phi]]: L^2(\R^2) \rightarrow L^2(\R^2)
$$
is bounded, and $\|[H_1, [H_2, \phi]]\|_{L^2 \rightarrow L^2} \approx
\|\phi\|_{\BMO}$.
\end{theorem}

From this result arises the question of the boundedness of iterated commutators on the endpoints. Our second fact then is the observation that the commutator $[b,H]$ is not in general bounded on $H^1(\mathbb R)$ for $b\in \BMO(\R)$ (see \cite{JPS}). It comes that if we are really considering the boundedness of these commutators on the endpoints ($H^1$ and $\BMO$), then we are only allowed to deal with functions with compact support. This is the idea in \cite{pottsehba1} from which comes our third fact that will be given after introducing several definitions and notations.

Let
  \begin{equation}\label{BMOrestrict}
 \BMO([0,1]^2):=\{f\in \BMO(\R^2): \supp f\subseteq [0,1]^2\}.
 \end{equation}
and
$$
\BMO^d([0,1]^2)= \{ f \in \BMO^d(\R^2): \supp f\subseteq [0,1]^2\}.
$$
 We say that $ f \in \LMO^d([0,1]^2) $, if $f \in \BMO^d([0,1]^2)$ and there exists $C>0$ with
$$
   \|Q_{\vk} f\|_{\BMO^d([0,1]^2)} \le C\frac{1}{(k_1+k_2+2)} \text{ for } \vk=(k_1, k_2) \in \N_0 \times \N_0.
$$
The spaces  $\LMO_{\vba}^d([0,1]^2) $ are defined correspondingly.
We also recall the notion of logarithmic mean oscillation introduced in \cite{pottsehba1}.

\vskip .2cm

We say that $ f \in \mathcal {LMO}^d([0,1]^2) $, if $f \in \BMO^d([0,1]^2)$ and there exists $C>0$ with
$$
   \|Q_{\vk} f\|_{\BMO^d([0,1]^2)} \le C\frac{1}{(k_1+1)(k_2+1)} \text{ for } \vk=(k_1, k_2) \in \N_0 \times \N_0.
$$

We next recall the relation between various spaces and their dyadic counterparts.
Given $\alpha=(\alpha_j)_{j\in \mathbb {Z}}\in \{0,1\}^{\mathbb {Z}}$ and
$r\in [1,2)$, we denote by $\mathcal {D}^{\alpha,r}=r\mathcal {D}^{\alpha}$ the dilated and translated standard dyadic grid $\mathcal {D}$ of $\R$ in the sense of \cite{hyt}. For
$\vec {\alpha}=(\alpha^1,\alpha^2)\in \{0,1\}^{\mathbb {Z}}\times \{0,1\}^{\mathbb {Z}}$ and $\vec {r}=(r_1,r_2)\in [1,2)^2$,
we define $\mathcal {D}^{\vec {\alpha},\vec {r}}$ to be the dilated and translated product dyadic grid in $\R^2$. This means that
$Q=Q_1\times Q_2\in \mathcal {D}^{\vec {\alpha},\vec {r}}$ if $Q_1\in r_1\mathcal {D}^{\alpha^1}$ and $Q_2\in r_2\mathcal {D}^{\alpha^2}$.

\vskip .2cm

Following  \cite{pipher}, \cite{treil}, we have the following relations between the strong notions of bounded mean oscillation and their dyadic versions.
\begin{multline*}
\BMO(\R^2)=\bigcap_{\vec {\alpha}\in \{0,1\}^{\mathbb {Z}}\times \{0,1\}^{\mathbb {Z}}, \vec {r}\in [1,2)^2 }\BMO^{d,\vec {\alpha}, \vec {r}}(\R^2)   \\
   =  \bigcap_{\vec {\alpha}\in \{0,1\}^{\mathbb {Z}}\times \{0,1\}^{\mathbb {Z}} }\BMO^{d,\vec {\alpha}, \vec {r_0}}(\R^2)  \text{ for any } \vec{r_0} \in [0,1)^2 ,
\end{multline*}

\begin{multline*}
\bmo(\R^2)=\bigcap_{\vec {\alpha}\in \{0,1\}^{\mathbb {Z}}\times \{0,1\}^{\mathbb {Z}}, \vec {r}\in [1,2)^2 }\bmo^{d,\vec {\alpha}, \vec {r}}(\R^2)   \\
   =  \bigcap_{\vec {\alpha}\in \{0,1\}^{\mathbb {Z}}\times \{0,1\}^{\mathbb {Z}} }\bmo^{d,\vec {\alpha}, \vec {r_0}}(\R^2)  \text{ for any } \vec{r_0} \in [0,1)^2 ,
\end{multline*}
where $\BMO^{d,\vec {\alpha}, \vec {r}}(\R^2)$ and $\bmo^{d,\vec {\alpha}, \vec {r}}(\R^2)$ are the dyadic (with respect to the product dyadic grid $\mathcal {D}^{\vec {\alpha},\vec {r}}$) $\BMO(\R^2)$ and $\bmo(\R^2)$ respectively.
One also obtains that
$$
\BMO([0,1]^2)=\bigcap_{\vec {\alpha}\in \{0,1\}^{\mathbb {Z}}\times \{0,1\}^{\mathbb {Z}}, \vec {r}\in [1,2)^2 }\BMO^{d,\vec {\alpha}, \vec {r}}([0,1]^2)
$$
and
$$
\bmo([0,1]^2)=\bigcap_{\vec {\alpha}\in \{0,1\}^{\mathbb {Z}}\times \{0,1\}^{\mathbb {Z}}, \vec {r}\in [1,2)^2 }\bmo^{d,\vec {\alpha}, \vec {r}}([0,1]^2).
$$
For the purpose of our last fact,  we introduce the space $\mathcal {LMO}([0,1]^2)$.
\begin{definition}\label{LMOrestrictps}
Let $f\in L^2(\R^2)$. We say that $f\in \mathcal {LMO}([0,1]^2)$ if $\supp f\subseteq [0,1]^2$,
 and there exists a constant $C>0$ such that for any $\vec {\alpha}\in \{0,1\}^{\mathbb {Z}}\times \{0,1\}^{\mathbb {Z}}$, $\vec {r}\in [1,2)^2$, and
$\vj = (j_1, j_2) \in \N_0\times \N_0$,
$$
\|Q^{\vec {\alpha}, \vec {r}}_{\vj}f\|_{\BMO^{d,\vec {\alpha}, \vec {r}}([0,1]^2)}\le C\frac{1}{(j_1+1)(j_2+1)}.
$$
\end{definition}
Here, $Q^{\vec {\alpha}, \vec {r}}_{\vj}$ denotes the projection as in (\ref{eq:qdef}), but relative to the dyadic grid $\mathcal{D}^{\vec {\alpha},\vec {r}}$; that is $$Q^{\vec {\alpha}, \vec {r}}_{\vj}f(s,t)=\sum_{r_1|I|\le 2^{-j_1}, r_2|J|\le 2^{-j_2}}\langle f,h_I^{\alpha_1,r_1}h_J^{\alpha_2,r2}\rangle h_I^{\alpha_1,r_1}(s)h_J^{\alpha_2,r2}(t),$$ where $h_I^{\alpha_l,r_l}$ is the Haar wavelet adapted to $I\in r_l\mathcal {D}^{\alpha_l}$, $l=1,2$.

\vskip .2cm
Our last fact is the following pretty recent result by S. Pott and the author \cite{pottsehba1}.
\begin{theorem} \label{thm:hankelps} Let $\phi
  \in \mathcal {LMO}([0,1]^2)$. Then
 $$
   [H_1, [H_2, \phi]]: \BMO([0,1]^2) \rightarrow \BMO(\R^2),
$$
is bounded, and $\|[H_1, [H_2, \phi]]\|_{\BMO([0,1]^2) \to \BMO(\R^2)} \lesssim
\|\phi\|_{\mathcal {LMO}([0,1]^2)}$.
\end{theorem}
We aim to replace in the last theorem, the space $\BMO([0,1]^2)$ by $\bmo([0,1]^2)$, keeping $\BMO(\R^2)$ as the target space. For this we will need to introduce the right concept of logarithmic mean oscillation here.
\begin{definition}\label{LMOrestrict}
Let $f\in L^2(\R^2)$. We say that $f\in \LMO([0,1]^2)$ if $\supp f\subseteq [0,1]^2$,
 and there exists a constant $C>0$ such that for any $\vec {\alpha}\in \{0,1\}^{\mathbb {Z}}\times \{0,1\}^{\mathbb {Z}}$, $\vec {r}\in [1,2)^2$, and
$\vj = (j_1, j_2) \in \N_0\times \N_0$,
$$
\|Q^{\vec {\alpha}, \vec {r}}_{\vj}f\|_{\BMO^{d,\vec {\alpha}, \vec {r}}([0,1]^2)}\le C\frac{1}{(j_1+j_2+2)}.
$$
\end{definition}

We do the following observations. First $\LMO([0,1]^2)$ continuously embeds into $\BMO([0,1]^2)$. Secondly,
 if we denote by $\LMO^{d,\vec {\alpha}, \vec {r}}([0,1]^2)$ the subset of
$\BMO^{d,\vec {\alpha}, \vec {r}}([0,1]^2)$ of functions $f$ such that there exists $C>0$ with
$$
    \|Q_{\vj}^{\vec {\alpha}, \vec {r}}f\|_{\BMO^{d,\vec {\alpha}, \vec {r}}([0,1]^2)}\le C\frac{1}{(j_1+j_2+2)}   \text{ for any }   \vj\in \N_0\times \N_0,
$$
then
 $$
 \LMO([0,1]^2)=\bigcap_{\vec {\alpha}\in \{0,1\}^{\mathbb {Z}}\times \{0,1\}^{\mathbb {Z}}, \vec {r}\in [1,2)^2 }\LMO^{d,\vec {\alpha}, \vec {r}}([0,1]^2).
 $$
$\LMO_1([0,1]^2)$ and $\LMO_2([0,1]^2)$ along with their dyadic counterparts are defined analogously. We are ready to give our main result of this section.
\begin{theorem} \label{thm:hankel} Let $\phi
  \in \LMO([0,1]^2)$. Then
 $$
   [H_1, [H_2, \phi]]: \bmo([0,1]^2) \rightarrow \BMO(\R^2),
$$
is bounded, and $\|[H_1, [H_2, \phi]]\|_{\bmo([0,1]^2) \to \BMO(\R^2)} \lesssim
\|\phi\|_{\LMO([0,1]^2)}$.
\end{theorem}

Before proving Theorem \ref{thm:hankel}, let us start by considering its dyadic counterpart. We introduce the dyadic shift operators $S^{d, {\alpha},  {r}}$, $\alpha\in \{0,1\}^{\mathbb Z}$, $r\in [1,2)$. These are the bounded linear operators $S^{d, {\alpha},  {r}}: L^2( \R) \rightarrow
L^2(\R)$  defined by $S^{d,{\alpha}, \vec {r}} h_I = h_{I^+} - h_{I^-}$, $I \in \D^{ {\alpha},  {r}}$. For simplicity, we restrict to the standard dyadic grid and write $S^{(1)} = S^d \otimes \eins$, $S^{(2)} = \eins \otimes S^d$, as operators on $L^2(\R^2) = L^2(\R) \otimes L^2(\R)$. The corresponding dyadic version of the above theorem is the following.

\begin{theorem} \label{thm:dyshift} Let $\phi \in \LMO^d([0,1]^2)$. Then
$$
   [S^{(1)}, [S^{(2)},\phi]]: \bmo^d([0,1]^2) \rightarrow \BMO^d(\R^2)
$$
is bounded, and $\|[S^{(1)}, [S^{(2)},\phi]]\|_{\bmo^d([0,1]^2) \rightarrow
  \BMO^d(\R^2)} \lesssim \|\phi\|_{\LMO^d([0,1]^2)}$.
\end{theorem}
\begin{proof}
Let us follow the ideas of \cite{pottsehba1,benoit}. First, we decompose the multiplication operator by $\phi$ into
nine parts: ${\Pi}_\phi$, ${\Delta}_\phi$,
${\Pi^{(0,1)}}_\phi$, ${\Pi^{(1,0)}}_\phi$, ${R_\Delta}_\phi$,
${R_\Pi}_\phi$, ${\Delta_R}_\phi$, ${\Pi_R}_\phi$, ${R_R}_\phi$,
corresponding to the matrix elements $\langle M_\phi h_I(s)
h_J(t), h_{I'}(s) h_{J'}(t) \rangle$ for $I' \subset I$, $I' = I$,
$I' \subset I$, $I' \supset I$, $J' \subset J$, $J' = J$, $J'
\supset J$.

Let us point out the symmetric pairs $({R_\Delta},{\Delta_R})$, $({R_\Pi},{\Pi_R})$, where
$$R_{R_\phi}b(s,t)=\sum_{I,J}b_{IJ}m_{IJ}(\phi)h_I(s)h_J(t),$$
$$\Pi_{R_\phi}b(s,t)=\sum_{I,J}m_J(\phi_I)m_I(b_J)h_I(s)h_J(t)$$ and
       $$\Delta_{R_\phi}b(s,t)=\sum_{I,J}m_J(\phi_I) b_{I,J} h_I(s)h_J^2(t).$$

The following was already proved in \cite{pottsehba1}.
\begin{lemma}\label{lemma:pottsehba1}
Let $\phi\in L^2([0,1]^2)$. Then the following estimates hold.
\begin{equation}
    \|[S^{(1)}, [S^{(2)},{R_R}_\phi]]\|_{\BMO^d([0,1]^2) \rightarrow \BMO^d(\R^2)} \le 2 \|\phi\|_{\BMO^d([0,1]^2)},
\end{equation}
\begin{equation}
  \|[S^{(1)}, [S^{(2)},{\Delta_R}_\phi]]\|_{\BMO^d([0,1]^2) \to \BMO^d(\R^2)} \lesssim \|\phi\|_{\BMO^d([0,1]^2)},
\end{equation}
\begin{equation}
  \|[S^{(1)}, [S^{(2)},{\Pi_R}_\phi]]\|_{\BMO^d([0,1]^2) \to \BMO^d(\R^2)} \lesssim \|\phi\|_{\LMO_{1}^d([0,1]^2)}.
\end{equation}
Swapping variables yields
\begin{equation}
  \|[S^{(1)}, [S^{(2)},{R_\Delta}_\phi]]\|_{\BMO^d([0,1]^2) \to \BMO^d(\R^2)} \lesssim \|\phi\|_{\BMO^d([0,1]^2)}
\end{equation}
and
\begin{equation}
  \|[S^{(1)}, [S^{(2)},R_{\Pi_\phi}]]\|_{\BMO^d([0,1]^2) \to \BMO^d(\R^2)} \lesssim \|\phi\|_{\LMO_{2}^d([0,1]^2)}.
\end{equation}
\end{lemma}

It comes that as $S^{(1)}$ and $S^{(2)}$ are bounded on
$\BMO^d(\R^2)$, it only remains to prove that for $\phi\in \LMO^d([0,1]^2)$, the operators ${\Pi}_\phi$, ${\Delta}_\phi$,
${\Pi^{(0,1)}}_\phi$, ${\Pi^{(1,0)}}_\phi$ are bounded from $\bmo^d([0,1]^2)$ to $\BMO^d(\R^2)$.

Again, we already have from \cite{pottsehba1} that
\begin{lemma}\label{lemma:pottsehbaotherpara}
Let $\phi \in L^2([0,1]^2)$. Then

\begin{enumerate}
\item $\Pi^{(1,1)}_\phi = \Delta_\phi: \BMO^d([0,1]^2) \rightarrow \BMO^d(\R^2)$ is bounded, if and only if $ \phi \in \BMO^d([0,1]^2)$.
Moreover, $$\|\Delta_\phi\|_{\BMO^d([0,1]^2) \rightarrow \BMO^d(\R^2)} \approx\|\phi\|_{\BMO^d([0,1]^2)}.$$
\item
$\Pi^{(1,0)}_\phi:\BMO^d([0,1]^2) \rightarrow \BMO^d(\R^2)$ is bounded if $ \phi \in \LMO_{1}^d([0,1]^2)$.
Moreover, $$\|\Pi^{(1,0)}_\phi\|_{\BMO^d([0,1]^2) \rightarrow \BMO^d(\R^2)}
  \lesssim \|\phi\|_{\LMO_{1}^d([0,1]^2)}.$$
\item
$\Pi^{(0,1)}_\phi:\BMO^d([0,1]^2) \rightarrow \BMO^d(\R^2)$ is bounded if $ \phi \in \LMO_{2}^d([0,1]^2)$.
Moreover, $$\|\Pi^{(0,1)}_\phi\|_{\BMO^d([0,1]^2) \rightarrow \BMO^d(\R^2)}
  \lesssim \|\phi\|_{\LMO_{2}^d([0,1]^2)}.$$
\end{enumerate}

\end{lemma}
Hence to finish the proof, we prove the following.
\begin{theorem}\label{thm:pararestrict}
Let $\phi \in L^2([0,1]^2)$. Then
 $\Pi^{(0, 0)}_\phi = \Pi_\phi:\bmo^d([0,1]^2) \rightarrow \BMO^d(\R^2)$ is bounded if  $ \phi \in \LMO^d([0,1]^2)$.
Moreover, $$\|\Pi_\phi\|_{\bmo^d([0,1]^2) \rightarrow \BMO^d(\R^2)} \lesssim \|\phi\|_{\LMO^d([0,1]^2)}.$$
\end{theorem}
\begin{proof}
An important ingredient of the proof is a local version of the first assertion in Lemma \ref{lemma:avgrowth}. For a bounded (not necessarily dyadic) interval $I \subset \R$, we define $s(I)$ as follows
$$
         s(I) =  \left\{ \begin{matrix} {\log|I|^{-1} +1}  & \text{ for } & |I| \le 1 \\
                                      1   & \text{ for } & |I| >1. \end{matrix}   \right.
$$
 For any $b\in \bmo([0,1]^2)$ and each rectangle $R = I \times J \subset \R^2$, as $m_Jb$ is uniformly in $\BMO([0,1])$, we get from the one parameter estimate of the mean of a function of bounded mean oscillation (see \cite{pottsehba1}) the following estimate.
 \begin{equation}\label{eq:avergloc}
  |m_R b|=|m_I(m_Jb)| \lesssim s(I) \|m_Jb\|_{\BMO([0,1])}\le s(I) \|b\|_{\bmo([0,1]^2)}. 
  \end{equation}
 Recall that we are looking to prove that given $\phi \in \LMO^d([0,1]^2)$, $b\in \BMO^d([0,1]^2)$ and  $f\in L^2(\mathbb R^2)$, the function
 $\Pi\left(\Pi(\phi,b),f\right)$ belongs to $L^2(\mathbb R^2)$.

\vskip .2cm

 We recall that $\D(\R)$ is the standard system of dyadic intervals in $\R$. The Haar basis of $L^2(\R^2)$ is
  $(h_I \otimes h_J)_{I, J \in \D(\R)} =   (h_R)_{R \in \D(\R) \times \D(\R)} $. We have the following decomposition (see also \cite{pottsehba1})
 $$
      f = \sum_{j_1 = - \infty}^\infty   \sum_{j_2 = - \infty}^\infty  \Da_{\vj} f                                        ,
 $$
 with
\begin{eqnarray*}
        \Da_{\vj} f &=& \sum_{|I |= 2^{-j_1}, |J|= 2^{- j_2} }   h_I(s) h_J(t) \langle f, h_I \otimes h_J \rangle    \\
           &=& \sum_{ R \in \D_{j_1}(\R) \times \D_{j_2}(\R)}   h_R  \langle f, h_R \rangle
 \end{eqnarray*}
$j_1, j_2 \in \mathbb Z.$

\vskip .2cm

Let $T:=\Pi\left(\Pi(\phi,b), \cdot \right)$. Then $T$ decomposes as follows
\begin{equation}\label{eq:Tdecomposition}
T=P_{(0,1)^2}T     +P_{(0,1)\times (0,1)^c}T   +  P_{(0,1)^c\times (0,1)} T   +P_{(0,1)^c\times (0,1)^c}T   ,
\end{equation}
where
\begin{eqnarray*}
 P_{(0,1)\times (0,1)}&=&   \sum_{j_1 = 0}^{\infty}\sum_{j_2=0}^{\infty}      \Da_{\vj},  \\
  P_{(0,1)\times (0,1)^c}&=&   \sum_{j_1 =0}^{\infty}\sum_{j_2=-\infty}^{-1}      \Da_{\vj}, \\
   P_{(0,1)^c\times (0,1)}&=&   \sum_{j_1 = - \infty}^{-1}\sum_{j_2=0}^{\infty}      \Da_{\vj}, \\
 P_{(0,1)^c\times (0,1)^c}&=&   \sum_{j_1 = - \infty}^{-1}\sum_{j_2=-\infty}^{-1}      \Da_{\vj}
\end{eqnarray*}
(see \cite{pottsehba1}).

Let us prove that each of the terms in the right hand side of the identity (\ref{eq:Tdecomposition}) is bounded on $L^2(\R^2)$.

We start with the last term. We observe that as
$$
P_{(0,1)^c\times (0,1)^c}\Pi\left(\Pi(\phi,b), \cdot \right)=\Pi\left(\Pi(P_{(0,1)^c\times (0,1)^c}\phi,b), \cdot \right),
$$
we only have to prove that given $\phi\in \LMO^d([0,1]^2)$ and $b\in \bmo^d([0,1]^2)$, $P_{(0,1)^c\times (0,1)^c}\Pi(\phi,b)$ belongs to $\BMO^d(\mathbb R^2)$. Observing with (\ref{eq:avergloc}) that for $R=I\times J\in \mathcal R$ with $|I|,|J|\ge 1$, $|m_Rb|\lesssim \|b\|_{\bmo^d([0,1]^2)}$, one obtains directly that for any open set $\Omega\subset \mathbb R^2$,
$$\|P_\Omega\left(P_{(0,1)^c\times (0,1)^c}\Pi(\phi,b)\right)\|_{L^2(\mathbb {R}^2)}^2\le \|P_\Omega\phi\|_{L^2(\mathbb {R}^2)}^2\|b\|_{\bmo^d([0,1]^2)}^2,$$
which proves that this term is bounded on $L^2(\mathbb R^2)$.

For $f\in L^2(\R^2)$, that
$$
      P_{(0,1)^2}\Pi\left(\Pi(\phi,b),f\right)   =      \Pi\left(\Pi(P_{(0,1)^2}       \phi,b),f\right)
$$
is in  $ L^2(\mathbb R^2)$ is obtained exactly as in the proof of Theorem \ref{thm:main}, with the help of the growth estimate (\ref{eq:avergloc}).

The second and third terms are symmetric, hence we only prove the boundedness of the second one. For this, we observe that as
$$P_{(0,1)\times (0,1)^c}\Pi\left(\Pi(\phi,b), \cdot \right)=\Pi\left(\Pi( P_{(0,1)\times (0,1)^c}      \phi,b), \cdot \right),$$
it is enough to prove that $\Pi( P_{(0,1)\times (0,1)^c}      \phi,b)\in \BMO(\R^2)$. The estimate (\ref{eq:avergloc}) tells us that for $R=I\times J\in \mathcal R$ with $|J|>1$, $|m_Rb|\lesssim \|b\|_{\bmo^d([0,1]^2)}$. Hence, for any open set $\Omega\in \R^2$,
$$\|P_\Omega\left(P_{(0,1)\times (0,1)^c}\Pi(\phi,b)\right)\|_{L^2(\mathbb {R}^2)}^2\le \|P_\Omega\phi\|_{L^2(\mathbb {R}^2)}^2\|b\|_{\bmo^d([0,1]^2)}^2,$$
which proves that this term is also bounded on $L^2(\mathbb R^2)$. The proof is complete.

\end{proof}
\end{proof}
We finish this section with the proof of Theorem \ref{thm:hankel}.
\begin{proof}[Proof of Theorem \ref{thm:hankel}]
The proof follows exactly as in \cite{pottsehba1} for the case of $[H_1,[H_2,b]]:\BMO([0,1]^2)\rightarrow \BMO(\R^2)$; we give it here for completeness. We use the fact that the Hilbert transform can be represented as
  averages of dyadic shifts (see \cite{hyt, pet}). This allows us to write for $b \in \bmo([0,1]^2)$ and $\phi \in
\LMO([0,1]^2)$,

$$
[H_1, [H_2, \phi]] b=\frac{64}{\pi^2}\int_1^2\int_1^2\int_{\{0, 1\}^{\mathbb {Z}}}\int_{\{0, 1\}^{\mathbb {Z}}}[S^{\alpha^1, r_1},[S^{\alpha^2, r_2}, \phi]] \, b \,d\mu(\alpha^1)\frac{dr_1}{r_1}\,d\mu(\alpha^2)\frac{dr_2}{r_2}
$$
(see \cite{hyt}). Next, we recall that for $b
\in \bmo([0,1]^2)$ and $\phi \in \LMO([0,1]^2)$, we have $b \in
\bmo^{d,\vec {\alpha}, \vec {r}}([0,1]^2)$, and $\phi \in
\LMO^{d,\vec {\alpha}, \vec {r}}([0,1]^2)$ for each $\vec {\alpha}=(\alpha^1, \alpha^2)\in \{0, 1\}^{\mathbb {Z}}\times \{0, 1\}^{\mathbb {Z}}$ and $\vec {r}=(r_1, r_2)\in [1,2)^2$ with uniformly bounded norms. 
It follows from Theorem \ref{thm:dyshift} that there exists a constant $C>0$ such
that
$$
     \|[S^{\alpha^1, r_1},[S^{\alpha^2, r_2}, \phi]] b\|_{\BMO^{d,\vec {\alpha}, \vec {r}}(\mathbb {R}^2)}
     \le C \|b\|_{\bmo([0,1]^2)} \|\phi\|_{\LMO([0,1]^2)} \text{ for all } (\alpha^1, \alpha^2, r_1, r_2).
$$
Hence, using \cite[Remark 0.5]{treil}, we obtain that
$$\frac{64}{\pi^2}\int_1^2\int_1^2\int_{\{0, 1\}^{\mathbb {Z}}}\int_{\{0, 1\}^{\mathbb {Z}}}[S^{\alpha^1, r_1},[S^{\alpha^2, r_2}, \phi]] \, b \,d\mu(\alpha^1)\frac{dr_1}{r_1}\,d\mu(\alpha^2)\frac{dr_2}{r_2}
  \in \BMO(\R^2)
$$
with norm controlled by $\|b\|_{\bmo([0,1]^2)} \|\phi\|_{\LMO([0,1]^2)}$. The proof is complete.
\end{proof}

\emph{Acknowledgements}: The author would like to thank his friend Willibrod Tetchi and his family for their kind hospitality during his stay in Cameroon where this work was completed.
\bibliographystyle{plain}

\end{document}